\documentclass[preprint]{article}
\usepackage[letterpaper, tmargin=15mm, bmargin=15mm, lmargin=15mm, rmargin=15mm]{geometry}
\usepackage{amssymb, amsthm, amsmath, amsfonts, euscript, qsymbols, delarray, latexsym, amstext, amscd, mathrsfs, mathtools}
\usepackage[utf8]{inputenc}
\usepackage{lmodern}
\usepackage[T1]{fontenc}
\usepackage{enumerate}
\usepackage{emptypage}
\usepackage[hidelinks]{hyperref}
\usepackage{fancyhdr}
\usepackage[all,cmtip]{xy}

\newtheorem{theorem}{Theorem}[section]

\newtheorem{prop}[theorem]{Proposition}
\newtheorem{coro}[theorem]{Corollary}

\newtheorem{ex}[theorem]{Example}

\newtheorem{ob}[theorem]{Observation}
\newcommand\from{\nobreak \hskip .1111em\mathpunct {}\nonscript \mkern-\thinmuskip {:}\hskip.3333emplus.0555em\relax}
\newcommand\id{\operatorname{id}}
\newcommand\reals{\mathbb R}

\newcommand\refArhangel{Arha}
\newcommand\refCollins{Collins}
\newcommand\refEngelking{Engelking}
\newcommand\refFlood{Flood}
\newcommand\refGabriyelyan{Gabriyelyan}
\newcommand\refKarnik{Karnik}
\newcommand\refMichael{Michael}
\newcommand\refOkunevb{Okunev1990}
\newcommand\refOkuneva{Okunev1995}
\newcommand\refHoshina{Hoshina2002}
\newcommand\refSanchez{Sanchez2014}
\newcommand\refSchaefer{Schaefer}
\newcommand\refSennott{Sennott}
\newcommand\refUspenskii{Uspenskii}
\newcommand\refYamazaki{Yamazaki}

\usepackage[dvipsnames]{xcolor}

\begin{document}
\title{$L$-retracts}
\author{Rodrigo Hidalgo Linares, Oleg Okunev\\
Benem\'erita Universidad Aut\'onoma de Puebla. Puebla, Mexico
}

\maketitle

\begin{abstract}
We study the relation of $L$-equivalence, which derives from the construction of the free locally convex spaces, 
through a concept that particularizes several notions related to the simultaneous extension of continuous functions. 
We also explore the relationship that this concept has with the Dugundji's extension theorem, and, based on this theorem we give sufficient conditions that allow us to identify these sets in different types of topological spaces. 
In particular, we present a method for constructing examples of $L$-equivalent mappings (and hence $L$-equivalent spaces)
that show that the properties of being an open or closed mapping are not $L$-invariant.
\end{abstract}

MSC-class: 46A03

\section{Introduction}

Let $\mathcal C$ be a class of topological spaces. A topological space $Y$ is called an absolute extender (AE) for the 
class $\mathcal{C}$ if, given a topological space $X$, a closed set $A\subset X$ and a continuous function 
$f\from A\to Y$, there is a continuous extension $F\from X\to Y$. In this sense, Dugundji's extension theorem tells us that 
locally convex spaces are absolute extenders for the class of metric spaces. With this in mind, consider a functor 
$\textbf{F}\from \textbf{Tych} \to \mathcal C$ that goes from the class of Tychonoff spaces to the class of 
topological spaces $\mathcal{C}$ and that assigns to each topological space $X$ a topological space $F(X)\in \mathcal{C}$ 
so that $F(X)$ contains a closed embedded copy of $X$ and so that every continuous function
$f\from X\to Y$, $Y\in \mathcal C$, has a continuous extension $f_\#\from F(X)\to Y$. If $A\subset X$ is such that
$F(A)$ is a subspace of $F(X)$, we will say that $A$ is $F$-embedded in $X$. Let us take a topological space $X$
and a subspace $A$, if any continuous function $f\from A\to Y$, $Y\in \mathcal C$, has a continuous extension 
$\tilde f\from X\to Y$, we will say that $A$ is an {\it $F$-valued retract} (or an {\it $F$-retract}) of $X$. Thus,
a continuous function $r\from X\to F(A)$ is called an {\it $F$-retraction}, or an {\it $F$-valued retraction}, if the 
restriction of $r$ to $A$ coincides with the embedding of $A$ in $F(A)$ and there is a continuous retraction 
$r_\#\from F(X)\to F(A)$ that extends $r$. In particular, $A$ is an {\it $L$-retract\/} of $X$ if $A$ is $L$-embedded in 
$X$ and the subspace $L(A)$ of $L(X)$ spanned by $A$ is a linear retract of $L(X)$.

Let us denote by $\textbf{HLocon}$ the category that has as objects the Hausdorff locally convex spaces and whose arrows are
the continuous linear mappings between them. Considering all of the above, we are especially interested in studying 
the $L$-valued retracts, where $\textbf{L}\from \textbf{Tych} \to \textbf{HLocon}$ is the functor that assigns to each
Tychonoff space $X$ its free locally convex space $L(X)$. 
This interest is largely motivated by the recent interest to the free locally convex spaces in current mathematical research.
In addition, as is well known, free locally convex spaces have a strong link with weak spaces of continuous functions, 
and although in general it is not possible to establish a natural topology $\eta$ on $C(X)$ so that $L(X)$ and $(C(X),\eta)$
be a dual pair just like $L_p(X)$ ($L(X)$ endowed with its $*$-weak topology) and $C_p(X)$, we can introduce concepts for
$C_p(X)$ motivated by concepts in the theory of $L(X)$. In particular, we will see that the $L$-retracts lead to a concept
in the $C_p$-theory that is stronger than the notion of an $\ell$-embedded set.

Specifically, these concepts are so similar that, based on them, we will try to carry out a study of the relation of 
$L$-equivalence in the same way that the relation of $\ell$-equivalence (which derives from the constructions of the weak 
spaces of continuous functions) has been investigated. In fact, we will see that the relation of $L$-equivalence can be
studied from the field of $C_p$-theory, and that this connection only implies imposing a minor extra condition. Although our
results focus on the $L$-equivalence of continuous mappings, we should keep in mind that this implies the $L$-equivalence of
topological spaces.

\section{Basic properties of free locally convex spaces}

In what follows, every topological space is assumed to be Tychonoff, that is, $T_1$ and completely regular. Likewise,
all topological vector spaces are assumed to be Hausdorff and are over $\reals$. The weak topological dual of a locally
convex spaces $E$ will be denoted by $E'$. We say that $E$ is {\it weak\/} if $E$ is topologically isomorphic to $(E')'$
(equivalently, the topology of $E$ is projective with respect to $E'$).

We define the free locally convex space (in the Markov sense) over a topological space $X$ as a pair $(\delta_X, L(X))$
formed by a continuous injection $\delta_X\from X\to L(X)$ and a locally convex space $L(X)$ such that $L(X)$ is the linear
span of $\delta_X(X)$  and for each continuous function $f\from X\to E$ to a locally convex space $E$ there is a continuous
linear mapping $f_\#\from L(X)\to E$ such that $f=f_\#\circ \delta_X$. Similarly to Graev, we define the free locally convex
space in the Graev sense over the topological space with a distinguished point $(X, x_{0})$ as a pair 
$(\delta_{X}, GL(X,x_0))$ formed by a continuous injection $\delta_X\from X\to GL(X,x_0)$ with $\delta_{X}(x_0)=0$ and
a locally convex space  $GL(X,x_0)$ such that $GL(X,x_0)$ is the linear span of $\delta_X(X)$ and for every continuous
function $f\from X\to E$ where $E$ is a locally convex space and $f(x_0)=0$, there is a unique continuous linear mapping 
$f_\#\from GL(X,x_{0})\to E$ such that $f=f_\#\circ \delta_X$.

The mapping $\delta_X$ is know as {\it the Dirac's embedding}, and for each $x\in X$ we have $\delta_X(x)=\delta_x$ is a 
linear functional that assigns to each $f\in \reals^X$ its value at $x$, that is, $\delta_x(f)=f(x)$. In this sense, we can view
the set $L(X)$ ($GL(X,x_0)$) as the set of finite linear combinations
$\lambda_1\delta_{x_1}+\dots +\lambda_n\delta_{x_n}$ with $n\in \mathbb N$, $\lambda_i\in \reals$ and 
$x_i\in X$ ($x_i\in X\setminus \{x_{0}\}$). The following facts are well known \cite\refGabriyelyan:

\begin{theorem}\label{Fundamental}
Let $X$ be a topological space and $x_{0}, x_1 \in X$ two different points. Then
\begin{enumerate}
\item[(1)] The spaces $L(X)$ and $GL(X,x_{0})$ always exist and are unique up to a topological isomorphism;
\item[(2)] $\delta_X(X)$ is a Hamel base for $L(X)$, and $\delta_X(X\setminus \{x_0\})$ is a Hamel base for $GL(X,x_{0})$;
\item[(3)] The topologies of $L(X)$ and $GL(X,x_{0})$ are Hausdorff and make the Dirac's embedding a topological
embedding, so that $X$ is embedded in $L(X)$ and $GL(X,x_{0})$ as closed subspace;
\item[(4)] For any $x_0, x_1\in X$, the spaces $GL(X,x_0)$ and $GL(X,x_1)$ are topologically isomorphic.
\end{enumerate}
\end{theorem}

To simplify notation, we will assume in what follows that $X$ is a subset of $L(X)$. The next statement is immediate from the
definition.

\begin{coro}
A linear mapping $f\from L(X)\to E$ to a locally convex space $E$ is continuous if and only if the restriction $f|X$
is continuous.
\end{coro}

In a categorical context, if we denote by $\textbf{Tych}_*$ the category of Tychonoff spaces with distinguished point and
continuous functions that preserve the distinguished points, Theorem \ref{Fundamental} tells us is that the forgetful
functors $\textbf{U}\from \textbf{HLocon}\to \textbf{Tych}$ and $\textbf{U}_{*}\from \textbf{HLocon}\to \textbf{Tych}_{*}$ have
left adjoint functors $\textbf{L}\from \textbf{Tych} \to \textbf{HLocon}$ and 
$\textbf{GL}\from \textbf{Tych}_* \to \textbf{HLocon}$. Note that there is also an adjunction 
$\textbf{V}\from \textbf{Tych}_*\to \textbf{Tych}$ and $\textbf{P}\from \textbf{Tych}\to \textbf{Tych}_*$, where 
$\textbf{V}$ is the forgetful functor and $\textbf{P}$ is the functor that assigns to each topological space $X$ the 
topological space $X^+=(X,a_X)$, where $a_X$ is an isolated point that does not belong to $X$, and to each continuous 
function $f\from X\to Y$ assigns $f^+\from X^+\to Y^+$ so that $f^+|X=f$ and $f^+(a_X)=a_Y$. Taking this into account, we
have the following results.

\begin{coro}\label{natiso}
The functors $\textbf{L}$ and $\textbf{GL}\circ \textbf{P}$ are naturally isomorphic; moreover, both $\textbf{L}$ and
 $\textbf{GL}$ respect finite coproducts.
\end{coro}

\begin{coro}\label{Coro-GL y P}
Let $X$ and $Y$ be topological spaces, $x_0$ a point of $X$, and $X\oplus Y$ their topological sum. Then
$GL(X\oplus Y, x_0)=GL(X,x_{0})\oplus L(Y)$.
\end{coro}

Let us show a more explicit relationship between $L(X)$ and $GL(X,x_0)$.
Consider the function $e_X\from X\to\reals$ such that $e_X(x)=1$ for all $x\in X$, and let 
$(e_X)_\#\from L(X)\to \reals$ be the unique linear mapping that extends $e_X$. Denote the kernel of 
$(e_X)_\#$ by $L^0(X)$. Observe that
\begin{equation*}
L^0(X)=\left\{ \displaystyle\sum_{i=1}^n\lambda_i\delta_{x_i}: n\in \mathbb{N},\quad \lambda_i\in \reals,
\quad x_i\in X, \quad 1\leq i\leq n, \quad\sum_{i=1}^n\lambda_i=0 \right\}.
\end{equation*}

We will say that a topological isomorphism $\varphi\from L(X)\to L(Y)$ is {\it special\/} if the composition 
$(e_Y)_\#\circ \varphi \from L(X)\to \reals$ is constant on $X$.

As shown in \cite{\refOkuneva}, if there is a topological isomorphism between $L(X)$ and $L(Y)$, then
there is always a special topological isomorphism between them. Moreover, the following statements are easily
derived from \cite[Theorem 3.7]\refOkuneva:

\begin{prop}\label{Prop-especial}
Given a topological isomorphism $\psi\from L(X)\to L(Y)$, there is always a topological isomorphism 
$\varphi\from L(X)\to L(Y)$ such that $(e_Y)_\#\circ \varphi =(e_X)_\#$.
\end{prop}

\begin{prop}
Let $x_0$ be a point of $X$. The spaces $L^0(X)$ and $GL(X,x_0)$ are topologically isomorphic.
\end{prop}

This proposition reflects the fact that the free locally convex space in the sense of Graev does not depend 
(up to a topological isomorphism) on the choice of the distinguished point. For this reason, in what follows
the free locally convex space in the sense of Graev will be denoted just by $GL(X)$.

\begin{coro}\label{CoroMarkovGraev}
Let $X$ and $Y$ be topological spaces. The spaces $L(X)$ and $L(Y)$ are topologically isomorphic
if and only if $GL(X)$ and $GL(Y)$ are topologically isomorphic.
\end{coro}

\begin{proof}
If the spaces $L(X)$ and $L(Y)$ are topologically isomorphic, then by Proposition \ref{Prop-especial}
there is a special topological isomorphism $\varphi\from L(X)\rightarrow L(Y)$ such that 
$(e_Y)_\#\circ \varphi = (e_X)_\#$. It follows that $\varphi|{L^0(X)}\from L^0(X)\to L^0(Y)$ is a topological
isomorphism, so the spaces $GL(X)$ and $GL(Y)$ are topologically isomorphic. On the other hand, if the spaces
$GL(X)$ and $GL(Y)$ are topologically isomorphic, then $L(X)=GL(X)\oplus\reals$ and 
$L(Y)=GL(Y)\oplus \reals$ are also topologically isomorphic.
\end{proof}

Given the close relationship between spaces $L(X)$ and $GL(X)$ we can define the relation of $L$-equivalence as follows: 
the spaces $X$ and $Y$ are called {\it $L$-equivalent\/} ( $X\stackrel{\mathit{L}}{\sim} Y$ ) if their free locally convex
spaces $L(X)$ and $L(Y)$ are topologically isomorphic. Furthermore, following \cite\refOkunevb, we can extend this relation
to continuous mappings between topological spaces. We say that two continuous mappings $f\from X\to Y$ and 
$g\from Z\to T$ are {\it $L$-equivalent\/} ($f\stackrel{\mathit{L}}{\sim} g$) if there are topological isomorphisms 
$\varphi \from L(X)\to L(Z)$ and $\psi\from L(Y)\to L(T)$ such that $\psi \circ f_\#=g_\#\circ \varphi$.
Clearly, these are equivalence relations. Likewise, any topological property of spaces or mappings that is preserved
by the relation of $L$-equivalence will be called {\it $L$-invariant\/}. It is worth noting that the $L$-equivalence
of the the identity mappings $\id_X\from X\to X$ and $\id_Y\from Y\to Y$ is the same as the $L$-equivalence of the spaces
$X$ and $Y$.

In a similar order of ideas, we can define the free weak topological vector space $L_p(X)$ over the topological space
$X$ as a pair $(\delta_X, L_p(X))$ formed by a continuous injection $\delta_X\from X\to L_p(X)$ and a weak topological
vector space $L_p(X)$ so that for every continuous function $f\from X\to E$ where $E$ is a weak topological vector space,
there is a unique continuous linear mapping $f_\#\from L_p(X)\to E$ such that $f=f_\#\circ \delta_X$. In addition,
Theorem \ref{Fundamental}, as well as the rest of subsequent statements, remain valid for these new spaces.

Naturally, this leads us to establish the concept of spaces and functions $L_p$-equivalent, and the notion of $L_p$-invariant
properties. It should be noted that the concept of $L_p$-equivalence is often linked to the functor $\mathbf C_p$, in which
case we say that two spaces $X$ and $Y$ are $\ell$-equivalent if its spaces of continuous real functions $C_p(X)$ and $
C_p(Y)$ are topologically isomorphic. This should not worry us, since the spaces $C_p(X)$ and $L_p(X)$ are in duality,
so $C_p(X)$ is topologically isomorphic to $C_p(Y)$ if and only if $L_p(X)$ is topologically isomorphic to $L_p(Y)$.
Therefore, following the notation already established, the relation of $L_p$-equivalence is the same as the relation of
$\ell$-equivalence, and the properties that are $L_p$-invariant are $\ell$-invariant.

Finally, we will briefly describe the relation between the topologies of the spaces $L(X)$ and $L_p(X)$. First,
from the definitions of these objects, it is easy to see that the identity $(\id_X)_\#\from L(X)\to L_p(X)$ is
a continuous linear mapping, accordingly, the underlying sets of the spaces $L(X)$ and $L_p(X)$ are the same, and it is
also clear that the topology of $L_p(X)$ is the $*$-weak topology of $L(X)$. Second, there is a relatively simple way
to describe its topology: since the spaces $L(X)$ and $C(X)$ are in duality (algebraic), and like any locally convex
topology over a space $E$ it is the topology of uniform convergence on the equicontinuous sets of its topological dual
$E'$, the topology of $L(X)$ it is the topology of uniform convergence on the equicontinuous pointwise bounded sets of
 $C(X)$ \cite{\refFlood}. Similarly, since that the topology of $L_p(X)$ it is weak, and since we can embed
$L_p(X)$ in $C_p(C_p(X))$, whose topology is also weak, we get that the topology of $L_p(X)$ is the topology inherited
from $C_p(C_p(X))$. Thus, a local neighborhood base of zero in $L(X)$ ($L_p(X)$) is is the family of sets of the form
\begin{equation}
\nonumber V[0,F, \varepsilon]=\left\{ \alpha\in L(X) : |\alpha(f)|= |f_\#(\alpha)|<\varepsilon, 
\quad f\in F \right\}.
\end{equation}
where $F\subset C(X)$ is an equicontinuous pointwise bounded set (respectively, a finite set) and $\varepsilon>0$,

\section{$L$-retracts}

As mentioned at the beginning, we will see what useful properties have the $\ell$-embedded sets and then we will try to 
translate them into the language of free locally convex spaces. We start with a definition: let $X$ be a topological space
and $Y$ a subset of $X$. An {\it extender} is a mapping $\phi\from C(Y)\to C(X)$ such that $\phi(f)|Y=f$ for every
$f\in C(Y)$. An extender may be linear or not, but what really matters to us is the situation in which it is continuous.
If there is a continuous (linear continuous) extender $\phi\from C_p(Y)\to C_p(X)$, we will say that $Y$ is
{\it $t$-embedded\/} ({\it $\ell$-embedded\/) in $X$}. A basic fact about $t$-embedded sets is that they always turn 
out to be closed. Clearly, every $\ell$-embedded set is also $t$-embedded, and it is easy to verify that $X$ is always 
$\ell$-embedded in $L_p(X)$. The following statement is also easy to prove.

\begin{prop}\label{Propl-encajaretract}
Let $Y$ be a subspace of $X$. The following statements are equivalent:
\begin{enumerate}
\item $Y$ is $\ell$-embedded in $X$;
\item There is a linear and continuous retraction $r\from L_p(X)\to L_p(Y)$;
\item There is a continuous functions $f\from X\to L_p(A)$ such that $f|A=\delta_{A}$;
\item Every continuous function from $Y$ to a weak topological linear space $E$ extends to a continuous function from
$X$ to $E$.
\end{enumerate}
\end{prop}

The previous proposition tells us that the $\ell$-embedded sets are simply the $L_p$-retracts, however, with the purpose of
not multiplying the notation, we will continue to use the term $\ell$-embedded set. On the other hand, with respect to
free locally convex spaces, if $Y$ is any subspace of $X$, it is not true that $L(Y)$ is necessarily a locally convex
subspace of $L(X)$, even if $Y$ is closed in $X$. Therefore, if $Y$ is a subspace of $X$ such that $L(Y)$ is a locally convex
subspace of $L(X)$, we will say that $Y$ is {\it $L$-embedded\/} in $X$.

In comparison, if $Y\subset X$, we have that $Y$ is $P$-embedded in $X$ if every continuous pseudometric on $Y$ can be
extended to a continuous pseudometric on $X$. The concept of a $P$-embedded set has several characterizations; the
one given by Yamazaki \cite[Theorem 3.1]\refYamazaki \ it is the one used in the proof of the following statement.

\begin{prop}
Let $Y$ be a subspace of $X$. The following statements are equivalents:
\begin{itemize}
\item $Y$ is $L$-embedded in $X$;
\item Any equicontinuous pointwise bounded subset of $C(Y)$ can be extended to an equicontinuous pointwise bounded
subset of $C(X)$;
\item $Y$ is $P$-embedded in $X$.
\end{itemize}
\end{prop}

Taking into account that the concept of an $L$-embedded set is related to simultaneous extension of equicontinuous
pointwise bounded sets, we can ask, of course, what relationship exists between the notions of an $\ell$-embedded set and
an $L$-embedded set.

\begin{ex}
{\sl An $L$-embedded sets need not be $\ell$-embedded.}

\smallskip
Consider the space $X=\omega_1+1$ with the order topology, and let $Y$ be the dense subspace $\omega_1$. Recall that $Y$
is a pseudocompact non-compact space, and that $X$ is a the Stone-\v Cech compactification of $Y$. Since the square of $Y$ is
pseudocompact, $X^2$ is the \v Cech-Stone compactification of $Y^2$, so $Y$ is $P$-embedded in $X$, that is, 
$Y$ is $L$-embedded in $X$. Since $Y$ it is not a closed set in $X$, $Y$ cannot be $\ell$-embedded. 
\end{ex}

\begin{ex}\label{l no L}
{\sl An $\ell$-embedded set need not be $L$-embedded.}

\smallskip
Let $Y$ be an uncountable discrete space, and let $X=L_p(Y)$. Then $Y$ is $\ell$-embedded in $X$,
and $X$ has the Souslin property. By \cite[Theorem 1.2]\refHoshina{},  $Y$ can not be $P$-embedded in
$X$, and hence $Y$ is not $L$-embedded in $X$.
\end{ex}

As we see, the $\ell$-embedded sets and free locally convex spaces do not have a direct relationship; this is 
another reason to study the $L$-retracts, because these have all the good properties of the $\ell$-embedded and 
$L$-embedded sets. As we will see, this combination of concepts improves their properties.

\begin{prop}\label{PropFL-encajaimpliLlB-encaja}
Every $L$-retract is an $L$-embedded and $\ell$-embedded set. In particular, every $L$-retract is a closed set.
\end{prop}

We still do not know if the inverse of the previous proposition holds, that is, whether every $L$-embedded
and $\ell$-embedded set is an $L$-retract. We only can guarantee the following.

\begin{theorem}\label{Teopreservadecuad}
Let $Y$ be a subspace of $X$. $Y$ is an $L$-retract of $X$ if and only if there is a continuous linear extender 
$\phi\from C_p(Y)\to C_{p}(X)$ such that if $B\subset C(Y)$ is an equicontinuous pointwise bounded set, 
then $\phi(A)$ also is an equicontinuous pointwise bounded set.
\end{theorem}

\begin{proof}
Suppose that $Y$ is an $L$-retract of $X$. Then there is a continuous linear retraction $r\from L(X)\to L(Y)$.
Define $\phi\from C_p(Y)\to C_p(X)$  by $\phi(f)= (f_\#\circ r)|X$ is a continuous linear extender, where 
 $f_\#\from L(Y)\to\reals$ is the linear extension of the function $f$ to $L(Y)$.
 
 Let $B\subset C(Y)$ be an equicontinuos pointwise bounded set; let us verify that the set 
 $\phi(B)=\{f_\# \circ r : f\in B \}$ is equicontinuous and pointwise bounded in $C(X)$.
By the definition of equicontinuity in a topological linear space \cite{\refSchaefer}, just note that given
an $\varepsilon>0$, the set 
\begin{equation*}
\bigcap_{f\in B} \left(f_\#\circ r\right)^{-1} (-\varepsilon,\varepsilon)
=r^{-1} \left( \bigcap_{f\in B} f^{-1}_\# (-\varepsilon,\varepsilon) \right)
\end{equation*}
is a neighborhood of zero. Thus, $\phi(B)$ is an equicontinuous pointwise bounded subset of $C(X)$.

It only remains to prove that if such a continuous linear extender exists, then $Y$ is a $L$-retract of $X$.
Define $q\from X\to L(Y)$ by $q(x)=\delta_x\circ\phi$ and let $r:L(X)\to L(Y)$ be the linear extension of $q$. Note that 
$q(x)$ is a continuous linear function on $C_p(Y)$, so $q(x)\in L_p(Y)$; thus, $q(x)$ also is an element of $L(Y)$, so $q$
is well-defined.  Furthermore, the restriction  $r|Y$ coincides with the Dirac's embedding of $Y$ in $L(Y)$, so $r$ is a
retraction. Thus, it remains only to verify that $q$ is continuous. 
Let $U=U[0,A,\varepsilon]$ be a neighborhood of zero in $L(Y)$, where $A\subset C(Y)$
 is an equicontinuous pointwise bounded set and $\varepsilon>0$. Since $\phi(A)$ is an equicontinuous pointwise bounded
 subset of $C(X)$, we have that the set $V=V[0,\phi(A), \varepsilon]$ is a neighborhood of zero in $L(X)$ such that
 $r (V)\subset U$. Since $r$ is continuous and linear, we conclude that $r \circ \delta_ X =q$ is continuous,
 so $Y$ is a $L$-retract of $X$.
\end{proof}

If $\varphi\from C_p(Y)\to C_p(X)$ is a continuous linear mapping such that for every equicontinuous pointwise bounded set
$A$ in $C(Y)$ the image $\varphi(A)$ is an equicontinuous pointwise bounded set in $C(Y)$, we will say that $\phi$ 
{\it preserves equicontinuous pointwise bounded sets}.

\begin{coro}
The spaces $X$ and $Y$ are $L$-equivalent if and only if there is a topological isomorphism 
$\varphi\from C_p(X)\to C_p(Y)$ such that both $\varphi$ and $ \varphi^{-1}$ preserve equicontinuous pointwise bounded sets.
\end{coro}

\begin{proof}
First let us suppose that $X$ and $Y$ are $L$-equivalent, that is, there is a topological isomorphism 
$\psi\from L(X)\to L(Y)$. Consider the mapping $\varphi \from  C_p(X)\to C_p(Y)$ defined by the rule 
$\varphi(f)=f_\# \circ \psi^{-1} \circ \delta_Y$. It is clear that $\varphi$ is continuous, linear and has
the inverse topological isomorphism $\varphi^{-1}(g)= g_\# \circ \psi \circ \delta_X$. It remains to show that given 
an $\varepsilon>0$ and an equicontinuous pointwise bounded set$A\subset C(X)$, the set
\begin{equation*}
\bigcap_{f\in A} \left(f_\# \circ \psi^{-1} \right)^{-1} (-\varepsilon,\varepsilon)
=\psi \left(\bigcap_{f\in A} {f}^{-1}_\# (-\varepsilon,\varepsilon) \right)
\end{equation*}
is a neighborhood of zero, but this is straightforward.

Conversely, if there is a topological isomorphism $\varphi\from C_p(X)\to C_p(Y)$ such that both 
$\varphi$ and $\varphi^{-1}$ preserves equicontinuous pointwise bounded sets, we can consider the map
$\psi\from L(X)\to L(Y)$ defined by $\psi(\alpha)=\alpha\circ\varphi^{-1}$. Recall that $\alpha\circ \varphi^{-1}$
is a continuous linear function on $C_p(X)$, so $\alpha\circ \varphi^{-1}$ is in $L_p(Y)$, and therefore in $L(Y)$.
Of course, $\psi$ has an inverse topological isomorphism given by $\psi^{-1}(\beta)=\beta \circ \varphi$.
Since both $\varphi$ and $\varphi^{-1}$ preserve equicontinuous pointwise bounded sets, both $\psi$ and $\psi^{-1}$
are continuous.
\end{proof}

The following statement only reinforces the known fact that in the class of $b_f$-spaces (such property is a
$\ell$-invariant) if two spaces are $\ell$-equivalent, then they are $L$-equivalent. Recall that a function 
$f\from X\to\reals$ is {\it $b$-continuous} if for every bounded set $A\subset X$ there is a continuous function
$g\from X\to\reals$ such that $g|A=f|A$. A space $X$ is called a {\it $b_f$-space} if every $b$-continuous real function 
is continuous. The class of $b_f$-spaces is larger than the class of $k$-spaces. Moreover, if $X$ is a $b_f$-space, 
a set $B\subset C_b(X)$ is compact if and only if $B$ is closed, equicontinuous and pointwise bounded \cite{\refUspenskii}.

\begin{coro}
Let $X$ and $Y$ be two $b_f$-spaces that are $\ell$-equivalent. Then $X$ and $Y$ are $L$-equivalent.
\end{coro}

\begin{proof}
Let $\varphi\from C_p(X)\to C_p(Y)$ be a topological isomorphism. It is not difficult to see that 
$\varphi\from C_b(X)\to C_b(Y)$ is a topological isomorphism ($C_b(X)$ is the space $C(X)$ endowed with the topology of
uniform convergence on the bounded sets of $X$). Now, let us take a set $A\subset C_p(X)$ which is equicontinuous and
pointwise bounded. Since $X$ is a $b_f$-space, we have that $[A]_b$, the closure of $A$ in $C_b(X)$, is compact.
Hence, $\varphi(A)\subset \varphi([A]_b)$ is equicontinuous and pointwise bounded. But $[A]_b=[A]_p$, the closure in 
$C_{p}(X)$, that is, $\varphi$ preserves equicontinuous pointwise bounded sets.
\end{proof}

\begin{ex}\rm
In particular, it is known that if $X$ is an uncountable discrete space, then the spaces $L_p(X)$ and $L_p(X)\oplus X$ are 
$\ell$-equivalent and they are not $L$-equivalent, that is, there is no topological isomorphism between $C_p(L_p(X))$ and
$C_p(L_p(X) \oplus X)$ that preserves equicontinuous pointwise bounded sets. On the other hand, each topological
isomorphism $\varphi\from C_p(X)\to C_p(Y)$, where $X$ and $Y$ are compact spaces, preserves equicontinuous pointwise
bounded sets.
\end{ex}

Returning to the consequences of Theorem \ref{Teopreservadecuad}, we have the following statement.

\begin{coro}\label{Criterio_L-retracto}
The following assertions are equivalent:
\begin{enumerate}
\item $Y$ is a $L$-retract of $X$;
\item There is a continuous linear retraction $r\from L(X)\to L(Y)$;
\item There is a continuous linear extender $\varphi\from C_p(Y) \to C_p(X)$ such that $\varphi$ preserves
equicontinuous pointwise bounded sets;
\item Every continuous function from $Y$ to a locally convex space $E$ extends to a continuous function form $X$ to $E$.
\end{enumerate}
\end{coro}

From this proposition it follows immediately that, in the same way that $X$ is $\ell$-embedded in $L_p(X)$ 
($X$ is an $\ell$-retract of $L_p(X)$), $X$ is an $L$-retract of $L(X)$. Also, note that in view of the 
Example \ref{l no L}, $X$ it is not always an $L$-retract of $L_{p}(X)$.

It is time to apply our results. First, based on Dugundji's extension theorem we have the following:

\begin{theorem}\label{TeoDugundji}
Let $X$ be a metric space. The following statements are equivalent:
\begin{itemize}
\item $Y$ is a closed subset of $X$;
\item $Y$ is a $L$-retract of $X$;
\item $Y$ is $\ell$-embedded in $X$.
\end{itemize}
\end{theorem}

\begin{proof}
Let $Y$ be a closed subset of a metric space $X$ and $\delta_Y\from Y\to L(Y)$ the Dirac's embedding of $Y$ in 
$L(Y)$. Applying the Dugundji's extension theorem we get a continuous function $f\from X\to L(Y)$ such that 
$f|Y=\delta_{Y}$, then, $Y$ is a $L$-retract of $X$. The other implications are clear.
\end{proof}

The Dugundji's extension theorem has been generalized in several ways, specifically, Borges generalized it to 
stratifiable spaces, and Stares did the same for the decreasing $(G)$ spaces (in the sense of \cite{\refCollins}).
On the other hand, note that each stratifiable space is a decreasing $(G)$ space, and each decreasing $(G)$ space is
hereditarily paracompact, so we could ask ourselves if for the hereditarily paracompact spaces, it is true that every
closed set is an $L$-retract. The answer is ``no".

\begin{ex}
{\sl There is a hereditarily paracompact space $X$ and a closed set $Y$ in $X$ such that $Y$ it is not a $L$-retract of $X$.}

\smallskip
{\rm Let $X$ be the Michael's line (\cite[Example 5.1.32]\refEngelking). In this space, the set $\mathbb Q$ of rational number
is a closed $P$-embedded set. On the other hand, the set $\mathbb P$ of irrational number with the topology inherited from
the Euclidean's metric, we have that the space of continuous functions with the compact-open topology $C_k(\mathbb{P})$ is a
locally convex space.  From \cite{\refSennott} we arrive at the existence of a continuous function 
$f\from \mathbb Q \to C_k(\mathbb P)$ that does not have continuous extension to $X$, namely, the function 
$f(x)(y)=\frac{1}{x-y}$, where $x\in \mathbb Q$ and $y\in \mathbb P$. With this we verify that $Y$ is not an $L$-retract of
$X$.}
\end{ex}

The previous example shows that, in general, we must impose stronger conditions on the subset $Y$ to make sure that 
$Y$ be an $L$-retract of $X$. For instance, we will see that some of them need metrizability as an additional condition.

A set $A\subset X$ is called {\it strongly discrete} if there is a discrete family $\{U_a: a\in A\}$ of disjoint open sets
in $X$ such that $a\in U_a$ for every $a\in A$. Taking into account the final observation of \cite{\refMichael} we easily get
the following.

\begin{coro}
Let $Y$ be a subspace of $X$. Then
\begin{enumerate}
\item If $X$ is paracompact and $Y$ is closed and metrizable, then $Y$ is an $L$-retract of $X$;
\item If $X$ is normal and $Y$ is closed, metrizable and separable, then $Y$ is an $L$-retract of $X$;
\item If $X$ is Tychonoff and $Y$ is compact and metrizable, then $Y$ is an $L$-retract of $X$;
\item If $X$ is Tychonoff and $Y$ is strongly discrete, then $Y$ is an $L$-retract of $X$.
\end{enumerate}
\end{coro}

\begin{proof}
The first three statements are obvious. In \cite{\refArhangel} it was shown that if $Y$ is a strongly discrete subspace,
then $Y$ is $\ell$-embedded in $X$. We will reproduce the original proof with the emphasis on the preservation by the defined
extender of equicontinuous pointwise bounded sets. Let $\mathcal U=\{U_y : y\in Y\}$ be a discrete family of disjoint
open sets in $X$ such that $y\in U_{y}$ for every $y\in Y$, also, for each $y\in Y$ let $h_y\in C(X)$ be a function 
such that $h_y(X)\subset [0,1]$), $h_y(y)=1$ and $h_y(X\setminus U_y)\subset \{0\}$. Define the function
$\psi(x)=\sum_{y\in Y}h_y(x)$. Since the family $\mathcal U$ is discrete, the function $\psi$ is defined on $X$ and is
continuous. Hence, the linear extender $\phi\from C_p(Y)\to C_{p}(X)$ defined by the rule 
$\phi(f)=\sum_{y\in Y}f(y)\cdot h_y$ is continuous.

Let $\mathcal F\subset C_p(Y)$ be an equicontinuous and pointwise bounded family of functions. We will verify that 
$\phi(\mathcal F)=\{\phi(f): f\in \mathcal F\}$ is equicontinuous and pointwise bounded. For each $y\in Y$, 
let $M_y\in\reals$ be such that $\{f(y):f\in \mathcal F\}\subset [-M_y, M_y]$. Given an $\varepsilon>0$ and $x\in X$,
if $x$ has a neighborhood disjoint from $\bigcup\mathcal U$, we have $\phi(f)(x)=0$
for every $f\in C_{p}(Y)$. Otherwise, there is a neighborhood $U$ of $x$ such that $U\cap U_y\neq\emptyset$ for
a unique $y\in Y$. Put $V=h_y^{-1}(h_{y}(x)-\varepsilon/M_y, h_y(x)+\varepsilon/M_y)$ and $W=U\cap V$. Then $W$ is an
open neighborhood of $x$, and for each $z\in W$ and $f\in\mathcal F$ we have
\begin{equation*}
\left | \phi(f)(x)-\phi(f)(z) \right |=\left | f(y)\left( h_{y}(x)-h_{y}(z)  \right)  \right |\leq M_{y} \left | h_{y}(x)-h_{y}(z) \right |< M_{y}\cdot \frac{\varepsilon}{M_{y}}=\varepsilon. 
\end{equation*}
that is, $\phi(\mathcal F)$ is an equicontinuous set that clearly is pointwise bounded.
\end{proof}

Note that, although in the class of metric spaces, the $L$-retracts and the $\ell$-embedded sets are the same, 
in general, in the generalizations of the Dugundji's extension theorem, we cannot weaken the condition that $Y$ is an
$L$-retract to being an $\ell$-embedded set.

\begin{ex}
\rm Let $Y$ be the discrete space of cardinality $\omega_1$ and $X=L_{p}(Y)$. It is clear that $Y$ is $\ell$-embedded in $X$.
The function $\delta_Y\from Y\to L(Y)$ has no continuous extension to $X$, because otherwise we would 
have that $Y$ is an $L$-retract of $X$, which is false ($Y$ is not $L$-embedded in $L_p(Y)$).

\smallskip 
Even if both $X$ and $Y$ are compact spaces, $Y$ need not be an $L$-retract of $X$. Indeed, let $X=\beta \mathbb{N}$ and
$Y=\beta \mathbb{N}\setminus \mathbb{N}$, then $Y$ is not $t$-embedded in $X$, that is, there is no continuous mapping
$\varphi \from C_p(Y)\to C_p(X)$ \cite\refArhangel.
\end{ex}

\section{A method for constructing examples of $L$-equivalent spaces}

Now we will concentrate on finding a method that generates examples of $L$-equivalent spaces. Clearly, the method described
by Okunev in \cite[Theorem 2.4]\refOkunevb{} already generates examples of $L$-equivalent spaces, however, the notion of a retract
is quite restrictive, and as we see, every retract is a $L$-retract. Thus, we will show that the notion of an $L$-retract is
sufficient for establishing our method.

Let $K_1$ and $K_2$ be two $L$-retracts of a space $X$, we will say that $K_1$ and $K_2$ are {\it parallel\/} if there are
continuous linear retractions $r_1\from L(X)\to L(K_1)$ and $r_2\from L(X)\to L(K_2)$ such that $r_1\circ r_2=r_1$ and
$r_2\circ r_1=r_2$. 

\begin{prop}
$K_1$ and $K_2$ are parallel $L$-retracts of $X$ if and only if there is a continuous linear retraction
$r_1\from L(X)\to L(K_1)$ such that the restriction $r_1|L(K_2)$ is a topological isomorphism from $L(K_2)$ onto $L(K_1)$.
In particular, $K_1$ is $L$-equivalent to $K_2$.
\end{prop}

\begin{proof}
Suppose $K_{1}$ and $K_{2}$ are parallel $L$-retracts of $X$. Let $r_1\from L(X)\to L(K_1)$ and $r_2\from L(X)\to L(K_2)$ be
continuous linear retractions such that $r_1\circ r_2=r_1$ and $r_2\circ r_1=r_2$.
Then $i=r_1|L(K_2)$ is a topological isomorphism of $L(K_2)$ onto $L(K_1)$ with the inverse $j=r_2|L(K_1)$.

Conversely, if there is a continuous linear retraction $r_1\from L(X)\to L(K_1)$ such that the restriction
$i=r_1|L(K_2)$ is a topological isomorphism from $L(K_2)$ onto $L(K_1)$, let $j=i^{-1}$ and put $r_2=j\circ r_1$.
Then $r_2$ is a continuous linear retraction from $L(X)$ to $L(K_2)$, $r_1\circ r_2=r_1$, and $r_2\circ r_1=r_2$.
\end{proof}

Recall that a continuous mapping $p\from X\to Y$ is {called $\reals$-quotient\/} if $p(X)=Y$ and whenever
 $f$ is a real function on $Y$ such that the composition  $f\circ p\from X\to \reals$ is continuous, $f$ is continuous
 \cite{\refKarnik}. The following statement is Proposition 1.1 in \cite{\refOkunevb}.
 
 \begin{prop}\label{rquot-tych}
 If $p\from X\to Y$ is an $\reals$-quotient mapping, $Z$ is a completely regular space, and $f\from Y\to Z$ is a function
 such that the composition $f\circ p$ is continuous, then $f$ is continuous.
 \end{prop}
 
 \begin{prop} \label{r-quot-open}
 A mapping $p\from X\to Y$ is $\reals$-quotient if and only if its extension $p_\#\from L(X)\to L(Y)$ is open.
 \end{prop}
 
 \begin{proof} Suppose that $p_\#$ is open, and let $f\from Y\to \reals$ be a function such that $f\circ p$ is continuous.
 Let $p_\#\from L(X)\to L(Y)$ and $f_\#\from L(Y)\to \reals$ be the linear extensions of $p$ and $f$. Then 
 $f_\#\circ p_\#=(f\circ p)_\#$ is continuous, and since $p_\#$ is open, $f_\#$ is continuous. Thus, $f=f_\#|Y$ is continuous.
 
 \smallskip
 Conversely, if $p$ is $\reals$-quotient, then, by the continuity,  the subspace $H=\ker p_\#$ is closed. Let $L=L(X)/H$ be the quotient space.
 The space $L$ is locally convex and Hausdorff, hence Tychonoff. Furthermore, there is a continuous bijection $i\from L\to L(Y)$
 such that $p_\#=i\circ \pi$ where $\pi\from L(X)\to L$ is the natural projection. Let us verify that the mapping
  $j=i^{-1}\from L(Y)\to L$ is continuous. It suffices to verify that the restriction $f=j|Y$ is continuous. We have
  $f\circ p=(j\circ p_\#)|X=\pi|X$, so $f\circ p$ is continuous; since $p$ is $\reals$-quotient, it follows that $f$ is
  continuous. Thus, $j$ is continuous, so $i$ is a topological isomorphism, and since $\pi$ is open, $p_\#$ is open.
  \end{proof}

There is a simple characterization of $L$-equivalence of $\reals$-quotient mappings.

\begin{prop}\label{criterio}
Two $\reals$-quotient mappings $f\from X\to Y$ and $g\from Z\to T$ are $L$-equivalent if and only if 
there is a topological isomorphism $i\from L(X)\to L(Z)$ such that $i(\ker f_\#)=\ker g_\#$.
\end{prop}

\begin{proof}
If $f$ and $g$ are $L$-equivalent, then there are topological isomorphisms $i\from L(X)\to L(Z)$ and 
$j\from L(Y)\to L(T)$ such that $j\circ  f_\#=g_\#\circ i$. Let $A=\ker f_\#$ and $B=\ker g_\#$. Then 
$j\circ f_\#(A)=\{0\}$, and by  $j\circ f_\#=g_\# \circ i$ we have to $\{0\}=g_\#\circ i(A)=g_\#(i(A))$, that is,
$i(A)\subset B$. For $g_\#=j\circ f_\#\circ i^{-1}$, we obtain that $\{0\}=g_\#(B)=j\circ f_\#\circ i^{-1}(B)$,
considering that $j$ is bijective we have that $f_\#\circ i^{-1}(B)=0$, hence, $i^{-1}(B)\subset A$, and this is enough
to establish the equality.

 \smallskip
Conversely, suppose that there is a topological isomorphism $i\from L(X)\to L(Z)$ such that $i(\ker f_\#)=\ker g_\#$.
Then there is an (algebraic) isomorphism $j\from L(Y)=L(X)/\ker f_\#\to L(T)=L(Z)/\ker g_\#$ such that
$j\circ f_\#=g_\#\circ i$. Since $g_\#$ and $i$ are continuous and $f_\#$ is open, $j$ is continuous.
Similarly, $j^{-1}\circ g_\#=f_\#\circ i^{-1}$, $f_\#$ and $i^{-1}$ are continuous, and $g_\#$ is open, so $j^{-1}$
is continuous. Thus, $i$ and $j$ are topological isomorphisms as required in the definition of $L$-equivalent mappings.
\end{proof}

Continuing with the $\reals$-quotient mappings, we will define the $\reals$-quotient spaces. Let $p\from X\to Y$
be a mapping of $X$ onto a set $Y$, It is known that there is a unique completely regular topology on the set $Y$ that makes $p$ a 
$\reals$-quotient mapping (this topology may be described as the weakest topology with respect to which all real-valued
functions on $Y$ with continuous compositions with $p$ are continuous). This topology is called the {\it $\reals$-quotient
topology}, and $Y$ endowed with this topology is the {\it $\reals$-quotient space with respect to the mapping $p$}
(or simply the {\it $\reals$-quotient space} if the mapping $p$ is clear from the context).
In this situation we say that  $p$ is {\it the natural mapping}.

Now, if $X$ if a space and $K$ is a closed set in $X$, let us denote $X/K=(X\setminus K) \cup \{K\}$, and let 
$p(x)=x$ for $x\in X\setminus K$, and $p(x)=K$ for each $x\in K$. Therefore, there is only one completely regular topology
on $X/K$ that makes it the $\reals$-quotient space with respect to $p$. It is shown in \cite{\refOkunevb} that this space
is Tychonoff. Also note that $p|(X\setminus K)\from X\setminus K \to X/K\setminus p(K)$ is a homeomorphism 
\cite[Corollary 1.7]\refOkunevb.

With all this we can establish our method:

\begin{theorem}\label{Principal}
If $K_1$ and $K_2$ are parallel $L$-retracts of $X$, then the $\reals$-quotient mappings $p_1\from X\to X/K_1$ and
 $p_2\from X\to X/K_2$ are $L$-equivalent. In particular, the spaces $X/K_1$ and $X/K_2$ are $L$-equivalent.
\end{theorem}

\begin{proof}
Let $r_1\from L(X)\to L(K_1)$ and $r_2\from L(X)\to L(K_2)$ be parallel $L$-retractions. We define a mapping
$i\from L(X)\to L(X)$ by the rule $i(\alpha)=r_1(\alpha)+r_2(\alpha)-\alpha$ for all $\alpha\in L(X)$. Clearly,
$i$ is linear and continuous. Moreover, $i\circ i(\alpha)=\alpha$, that is, $i$ is its own inverse, so $i$ is a topological
isomorphism. Let us put $s_2=r_2|L(K_1)$, then $s_2$ is a topological isomorphism such that $s_2\circ r_1=r_2\circ i$.
It follows that $i(L(K_1))=L(K_2)$ and that $i(\ker r_1)=\ker r_2$.

Clearly, $\ker (p_i)_\#=L^0(K_i)=\ker (e_{K_i})_\#$, $i=1,2$. Since $K_1$ and $K_2$ are $L$-equivalent, there is
a special topological isomorphism $k\from L(K_1)\to L(K_2)$ such that $(e_{K_2})_\#\circ k=(e_{K_1})_\#$. Let
$g=k\times j$ where $j=i|\ker r_1$. If $\alpha\in L(X)$, then  the mappings $\eta_i\from L(X)\to L(K_i)\times \ker r_i$,
$i=1,2$, defined by $\eta_i(\alpha)=(r_i(\alpha),\alpha-r_i(\alpha))$ are topological isomorphism whose inverses
are $\xi_i\from L(K_i)\times \ker r_i\to L(X)$ are $\xi{}(\alpha,\beta)=\alpha+\beta$, $i=1,2$. Defining a mapping $\psi$
by
\begin{equation*}
\psi(\alpha)=\zeta_2\circ g\circ \eta_1(\alpha)=\zeta_2\circ g(r_1(\alpha),\alpha-r_1(\alpha))
=\zeta_2(k(r_1(\alpha)), j(\alpha-r_1(\alpha)))=k(r_1(\alpha))+j(\alpha)-j(r_1(\alpha)),
\end{equation*}
we obtain a topological isomorphism such that $\psi(L^0(K_1))=L^0(K_2)$. Thus, by Proposition \ref{criterio}, 
$p_1$ is $L$-equivalent to $p_2$.
\end{proof}

\begin{coro}\label{Corolario20}
Let $X$ be a topological space and $K\subset X$ an $L$-retract of $X$. Then the spaces $X^+$ and $X/K\oplus K$ are 
$L$-equivalent.
\end{coro}

\begin{proof}
Let $K'$ be a homeomorphic copy of $K$ disjoint from $X$ and $\varphi\from K_1\to K'$ a homeomorphism.
Put $Z=X\oplus K'$, Then $L(Z)$ is topologically isomorphic to $L(X)\oplus L(K')$.
Let $r\from L(X)\to L(K)$ be an $L$-retraction. Define $r_1\from L(Z)\to L(K)$ by putting $r_1|L(X)=r$ and 
$r_1|L(K')=\varphi_\#^{-1}$ and $r_2\from L(Z)\to L(K')$ by putting $r_2|L(X)=\varphi_\#\circ r$ and 
$r_2| L(K')=\id_{L(K')}$. Then $(r_1\circ r_1)| L(X)=r_1\circ r=r=r_1|L(X)$ and $(r_1\circ r_1)| L(K')=
r_1\circ \varphi_\#^{-1}=\varphi_\#^{-1}$ (because $r_1|L(K)$ is the identity), so $(r_1\circ r_1)|L(K')=r_1|L(K')$.
We conclude that $r_1\circ r_1=r_1$, so $r_1$ is a retraction. Similarly, $(r_2\circ r_2)|L(X)=r_2\circ \varphi_\#\circ r|L(X)
=\varphi_\# \circ r=r_2|L(X)$, because $r_2|L(K')$ is the identity, and $(r_2\circ r_2)| L(K')=r_2|L(K')$. Thus, 
$r_2\circ r_2=r_2$, and $r_2$ is a retraction.

Furthermore, $(r_1\circ r_2)|L(X)=r_1\circ \varphi\circ r=\varphi^{-1}\circ\varphi\circ r=r=r_1|L(X)$ and 
$(r_1\circ r_2)|L(K')=r_1$ because $r_2|L(K')$ is the identity. Thus, $r_1\circ r_2=r_1$. Similarly, 
$(r_2\circ r_1)|L(X)=r_2\circ r=\varphi_\#\circ r=\varphi_\#\circ r=r_2|L(X)$ and 
$(r_2\circ r_1)|L(K')=(r_2|L(K))\circ\varphi_\#^{-1}=(\varphi_\#\circ r)|L(K)\circ\varphi_\#^{-1}=
\varphi_\#\circ\varphi_\#^{-1}=\id_{L(K')}=r_2|L(K')$, so $r_2\circ r_1=r_2$. Thus $r_1$ and $r_2$ are parallel
$L$-retractions. By Theorem \ref{Principal}, the spaces $Z/K$ and $Z/K'$ are $L$-equivalent. Clearly, $Z/K$ is homeomorphic 
to $X/K\oplus K$ and $Z/K'$ is homeomorphic to $X^+$.
\end{proof}

Note that in the proof of Theorem \ref{Principal}, the fact that the $L$-retracts are parallel served to guarantee 
the existence of a pair of topological isomorphisms $s_2$ and $i$ such that $s_2\circ r_1=r_2\circ i$.
Therefore, in the case that two sets $K_1$ and $K_2$ are $L$-retracts of $X$ and there are topological isomorphisms
$i\from L(X)\to L(X)$, $j\from L(K_1)\to L(K_2)$ and continuous linear retractions $r_1\from L(X)\to L(K_1)$ and 
$r_2\from L(X)\to L(K_2)$ such that $j\circ r_1=r_2\circ i$ we will say that these sets are {\it equivalent $L$-retracts}.

\begin{prop}\label{PropGLigualkernel}
Let $r\from L(X)\to L(K)$ be a continuous linear retraction, where $K\subset X$. Then $L(X)$ is topologically isomorphic to
$GL(X/K)\times L(K)$, and $GL(X/K)$ is topologically isomorphic to $\ker r$.
\end{prop}

\begin{proof}
The first part follows from the fact that $X^{+}$ is $L$-equivalent to $X/K \oplus K$, therefore, applying 
Corollary \ref{CoroMarkovGraev} we obtain that $GL(X^{+})$ is topologically isomorphic to $GL(X/K\oplus K)$, 
that is $L(X)$ is topologically isomorphic to $GL(X/K)\oplus L(K)$ (Corollary \ref{Coro-GL y P}).

We will write  $L\cong E$ if the topological linear spaces $L$ and $E$ are topologically isomorphic.
The second part is due to the observation that if $r\from L(X)\to L(K)$ is a continuous linear retraction, then
$L(X)\cong L(K)\times \ker r$. Thus, $L(K)\times \ker r \cong L(K)\oplus GL(X/K) \cong L(K)\times GL(X/K)$. To end the proof,
note that the function $\theta\from X\to X/K\subset GL(X/K)$ given by $\theta(x)=p(x)$ is $\reals$-quotient, so
$\theta_\#\from L(X)\to GL(X/K)$ is open and onto. Since $\ker \theta_\#=L(K)$, we have $L(X)/L(K)\cong GL(X/K)$.
On the other hand, the function $\psi\from L(X)\to \ker r$ given by $\psi(\alpha)=\alpha-r(\alpha)$ is linear continuous,
open and its kernel is $L(K)$.  Thus, $L(X)/L(K)\cong \ker r$. We conclude that $GL(X/K)\cong \ker r$.
\end{proof}

In a way, given an $L$-retraction $r\from L(X)\to L(K)$, we can obtain enough information about $L(X)$ from $L(K)$,
then, as a corollary of the previous proposition, we can obtain that $L(X)$ is topologically isomorphic
to $\ker p_\#\oplus \ker r \oplus \reals$, where $p_\#$ it is the linear continuous extension of the natural mapping
$p\from X\to X/K$. This motivates the following statements:

\begin{prop}\label{p-equiv}
Let $K_1$ and $K_2$ be two $L$-retracts of $X$. If the natural mappings $p_1\from X\to X/K_1$ and $p_2\from X\to X/K_2$
are $L$-equivalent, then $K_1$ and $K_2$ are equivalent $L$-retracts.
\end{prop}

\begin{proof}
Since the natural mappings $p_1$ and $p_2$ are $L$-equivalent, there are topological isomorphisms $i\from L(X)\to L(X)$ and
$j\from L(X/K_1)\to L(X/K_2)$ such that $j\circ (p_1)_\#=(p_2)_\#\circ i$. In view of the assumption that $X/K_1$ is 
$L$-equivalent to $X/K_2$, we have that $\ker r_1$ is topologically isomorphic to $\ker r_2$; let us denote by $t$ 
such a topological isomorphism. Also, from $i(L^0(K_1))=i(\ker (p_1)_\#)=\ker (p_2)_\#=L^0(K_2)$, we obtain that $L(K_1)$ is
topologically isomorphic to $L(K_2)$; let $k$ be such an isomorphism. Then $w=k\times t$ is a topological isomorphism
between $L(K_1)\times \ker r_1$ and $L(K_2)\times \ker r_2$. Thus, we have a topological isomorphism $\varphi\from L(X)\to L(X)$
given by the formulas:
\begin{equation*}
\varphi(a)=\zeta_2\circ w\circ \eta_1(a)=\zeta_2\circ w(r_1(a),a-r_{1}(a))=\zeta_2(k(r_1(a)), t(a)-t(r_1(a))=k(r_1(a))+
t(a)-t(r_1(a)).
\end{equation*}

We quickly notice that under this isomorphism, $\varphi(\ker r_1)=\ker r_2$, so, defining $\phi\from L(K_1)\to L(K_2)$
by $\psi(a)=r_2\circ \phi (r_1^{-1}(a))$ we obtain a topological isomorphism such that $\psi\circ r_1=r_2\circ \phi$,
proving that $K_1$ and $K_2$ are equivalent $L$-retracts.
\end{proof}

\begin{coro}
Let $K_1$ and $K_2$ be $L$-retracts of $X$, and $p_1\from X\to X/K_1$, $p_2\from X\to X/K_2$ the corresponding natural
mappings. The following statements are equivalent:
\begin{enumerate}
\item $K_1$ and $K_2$ are equivalent $L$-retracts;
\item $p_1$ and $p_2$ are $L$-equivalent;
\item $K_1$ is $L$-equivalent to $K_2$, and $X/K_1$ is $L$-equivalent to $X/K_2$.
\end{enumerate}
\end{coro}

\begin{proof}
The equivalence between items 1 and 2 is obvious. That 1 implies 3 is easy to verify. Therefore, we will only prove
that the statement 3 implies the statement 1. First, by the hypothesis, we have $GL(X/K_1)\cong GL(X/K_2)$, and accordingly,
due to Proposition \ref{PropGLigualkernel} we have that $\ker r_1$ is topologically isomorphic to $\ker r_2$. Then, using the
technique described in the previous propositions, we obtain topological isomorphisms $i\from L(X)\to L(X)$ and
$j\from L(K_1)\to L(K_2)$ such that $i(\ker r_1)=\ker r_2$ and $j\circ r_1=r_2\circ i$. 
It follows that $K_1$ and $K_2$ are equivalent $L$-retracts.
\end{proof} 

\begin{coro}
Let $r_1\from X\to K_1$ and $r_2\from X\to K_2$ be retractions in $X$, and $p_1\from X\to X/K_1$, $p_2\from X\to X/K_2$
the natural mappings. The following statements are equivalent:
\begin{enumerate}
\item $r_1$ is $L$-equivalent to $r_2$;
\item $p_1$ is $L$-equivalent to $p_2$.
\item $K_1$ is $L$-equivalent to $K_2$, and $X/K_1$ is $L$-equivalent to $X/K_2$.
\end{enumerate}
\end{coro}

\begin{coro}
Two retractions to the same retract are $L$-equivalent.
\end{coro}

\begin{coro}
Let $X$ and $Y$ be two $L$-equivalent spaces, $K_1$ and $K_2$ be retracts respectively of of $X$ and $Y$, 
which are $L$-equivalent and such that $X/K_1$ is $L$-equivalent to $Y/K_2$. Then any two retractions
$X\to K_1$ and $X\to K_2$ are $L$-equivalent, moreover, the corresponding natural mappings are also $L$-equivalent.
\end{coro}

\begin{ex} \rm Consider the retractions $r_1, r_2\from [0,1]\to [0,1/2]$ defined by $r_1(x)=x$ if $x\in [0,1/2]$ and 
$r_1(x)=1-x$ if $x\in [1/2,1]$, $r_2(x)=x$ if $x\in [0,1/2]$ and $r_2(x)=1/2$ if $x\in [1/2,1]$. Theses retractions
are $L$-equivalent and we have that $r_1$ is perfect and open, while $r_2$ it is also perfect but not open.
\end{ex}

\begin{coro}
Being an open mapping is not preserved under the relation of $L$-equivalence, even within the class of perfect retractions.
\end{coro}

\begin{ex}\rm
In \cite[Theorem 4.2]\refSanchez{} it is shown that there is a weakly pseudocompact, not locally compact space
that has a single non-isolated point, of the form $K=[0,\omega_1]\cup \{A_\alpha: \alpha\in\omega_1\}$. Moreover,
it is proved that $K$ is $L$-equivalent to $Y=K_1\oplus K_2$, where $K_1=[0,\omega_1]$ and $K_2=(K\setminus K_1)\cup 
\{\omega_1\}$. Let $Z$ be a pseudocompact space that contains $Y$ as a closed subspace. We consider the spaces 
$X_1=K\oplus Z$ and $X_2=Y\oplus Z$, and the retractions $r_1\from X_1 \to Z$ such that $r_1|K$ is the embedding of $K$
in $Z$ ($K$ can be seen as a subspace of $Y$) and $r_1|Z$ is the identity; $r_2\from X_2 \to Z$ such that $r_2|Y$ is the
embedding of $Y$ in $Z$ and $r_2|Z$ is the identity. This retractions are $L$-equivalent and they are finite-to-one,
but $r_1$ is perfect and $r_2$ it is not (it is not closed).
\end{ex}

\begin{coro}
Being a closed mapping is not preserved under the relation of $L$-equivalence, even within the class of 
finite-to-one retractions. In particular, being a perfect function is not $L$-invariant.
\end{coro}

\begin{ob}\rm
Note the following relationship: each pair of parallel retracts are parallel $L$-retracts, and therefore they are 
equivalent $L$-retracts. On the other hand, we know that there are parallel $L$-retracts that are not parallel retracts,
in fact, let $K_1$ and $K_2$ be two $L$-equivalent spaces that are not homeomorphic, then $X=K_1\oplus K_2$ contains both
spaces as parallel $L$-retracts that clearly are not parallel retracts. However, the following question arises:
{\sl are two equivalent $L$-retracts always parallel $L$-retracts?}
\end{ob}

\bibliographystyle{abbrv} 

\end{document}